\documentclass[10pt,a4paper]{article}
\pdfoutput=1

\usepackage{amssymb}
\usepackage{amsfonts}
\usepackage{amsmath}
\usepackage[amsmath,thmmarks]{ntheorem}
\usepackage{mathrsfs}
\usepackage[cmtip,arrow]{xy}
\usepackage{pb-diagram,pb-xy}
\usepackage[draft]{fixme}
\usepackage{graphicx}

\newcommand{\R}{\mathbb{R}}
\newcommand{\C}{\mathbb{C}}
\newcommand{\Z}{\mathbb{Z}}
\newcommand{\im}{\textrm{Im}}
\newcommand{\id}{\textrm{id}}
\newcommand{\K}{\mathscr{K}}
\newcommand{\EK}{\mathscr{EK}}
\newcommand{\A}{\mathscr{A}}
\newcommand{\pr}{\mathrm{pr}}

\newcommand{\N}{\mathbb{N}}
\newcommand{\HH}{\mathscr{H}}

\newtheorem{thm}[equation]{Theorem}
\newtheorem{lem}[equation]{Lemma}
\newtheorem{prop}[equation]{Proposition}
\newtheorem{cor}[equation]{Corollary}
\newtheorem{conj}[equation]{Conjecture}

\theoremstyle{plain}
\theorembodyfont{\normalfont}
\newtheorem{defn}[equation]{Definition}
\newtheorem{ex}[equation]{Example}
\newtheorem{rem}[equation]{Remark}

\theoremstyle{nonumberplain}
\theoremheaderfont{\normalfont\bfseries}
\theorembodyfont{\normalfont}
\theoremsymbol{\ensuremath{\square}}
\theoremseparator{.}
\newtheorem{proof}{Proof}

\theoremstyle{plain}
\theorembodyfont{\itshape}
\theoremsymbol{\ensuremath{\square}}

\title{The structure of groups of multigerm equivalences}

\author{Aasa Feragen\footnote{1) Department of Computer Science, University of Copenhagen, Universitetsparken 1, 2100 Copenhagen, Denmark; aasa@diku.dk.}
 \and Andrew du Plessis\footnote{ Department of Mathematical Sciences, Aarhus University, Ny Munkegade 118, Bldg 1530, 8000 Aarhus, Denmark.}
}

\begin{document}

\maketitle

\begin{abstract}
We study the structure of classical groups of equivalences for smooth multigerms $f \colon (N,S) \to (P,y)$, and extend several known results for monogerm equivalences to the case of mulitgerms. In particular, we study the group $\A$ of source- and target diffeomorphism germs, and its stabilizer $\A_f$. For monogerms $f$ it is well-known that if $f$ is finitely $\A$-determined, then $\A_f$ has a maximal compact subgroup $MC(\A_f)$, unique up to conjugacy, and $\A_f/MC(\A_f)$ is contractible. We prove the same result for finitely $\A$-determined multigerms $f$. Moreover, we show that for a ministable multigerm $f$, the maximal compact subgroup $MC(\A_f)$ decomposes as a product of maximal compact subgroups $MC(\A_{g_i})$ for suitable representatives $g_i$ of the monogerm components of $f$. We study a product decomposition of $MC(\A_f)$ in terms of $MC(\mathscr{R}_f)$ and a group of target diffeomorphisms, and conjecture a decomposition theorem. Finally, we show that for a large class of maps, maximal compact subgroups are small and easy to compute.
%\keywords{singularities, groups of diffeomorphism germs \and symmetries of multigerms \and maximal compact subgroup \and contractible quotient \and right-left equivalence \and $\A$-equivalence}
% \PACS{PACS code1 \and PACS code2 \and more}
% \subclass{MSC code1 \and MSC code2 \and more}
\end{abstract}

\section{Introduction}

When studying global properties of smooth, singular maps, a typical approach is to first solve problems locally, and then glue the local solutions together. When the map $F$ in question is stable, the typical local situation is the following: Singularities of $F$ are found along stratified subsets of source and target of $F$. Each stratum consists of points where a certain singular germ $f$ appears as a singularity of $F$. These strata are submanifolds of source and target, respectively, and in tubular neighborhoods of these submanifolds, $F$ restricts to a fibered family of germs $f$, as illustrated in fig.~\ref{tubedmap}. Such decompositions are discussed in detail in~\cite{feragenthesis}. 

The singular germs are determined up to changes of local coordinates in source and target of the germ, also called a \emph{right-left equivalence}, or, in modern terminology, an $\A$-equivalence. Two germs $f, g \colon (\R^n,0) \to (\R^p,0)$ are right-left-equivalent, and hence define the same singularity, if there exist diffeomorphism germs $\phi \colon (\R^n,0) \to (\R^n,0)$ and $\psi \colon (\R^p,0) \to (\R^p,0)$ such that $g = \psi \circ f \circ \phi$. The group $\A$ of all such equivalences $(\phi, \psi)$ coincides with the product of groups $\mathscr{R} \times \mathscr{L}$, where $\mathscr{R}$ consists of source diffeomorphism germs $\phi$ and $\mathscr{L}$ consists of target diffeomorphism germs $\psi$. These groups all act on the space of multigerms $f \colon (\R^n,0) \to (\R^p,0)$ through composition of maps, e.g., $(\phi, \psi) \cdot f = \psi \circ f \circ \phi$.

\begin{figure}
\centering
\includegraphics[height=3.5cm]{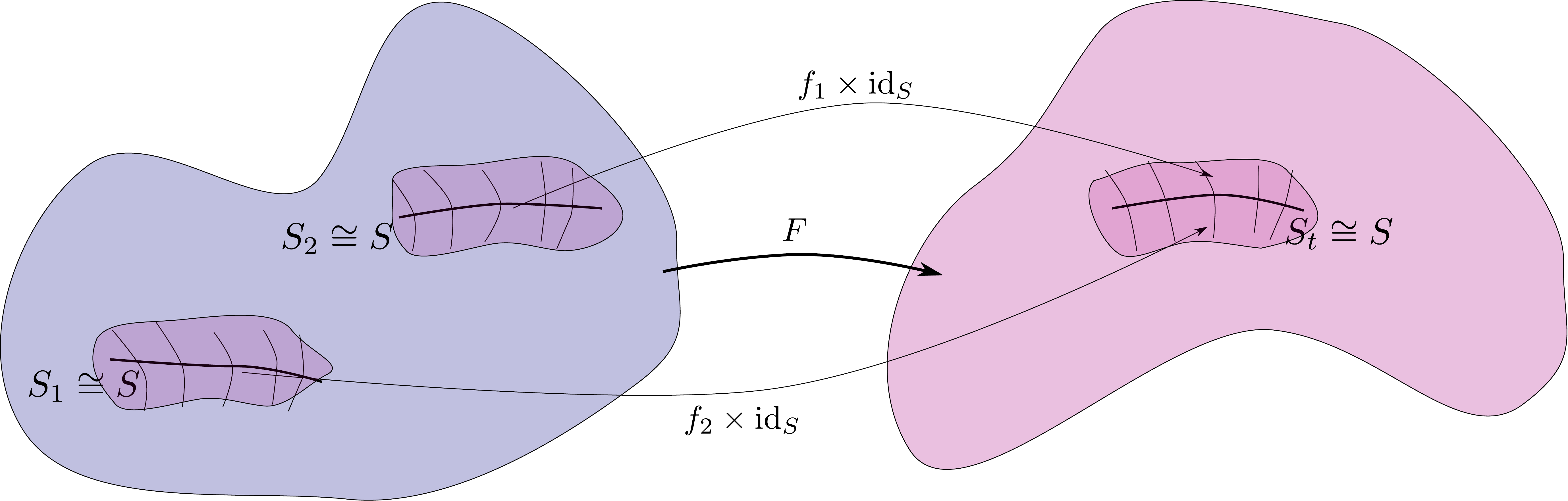}
\caption{Along the two submanifolds $S_1$ and $S_2$ of the source manifold $N$, we find singularities $f_1$ and $f_2$, which are both mapped onto the submanifold $S_t$ of the target manifold $P$. If $F$ is stable, then in tubular neighborhoods $T_1$ and $T_2$, $F$ will restrict to a fibered map $f \times \id_S$, where $f = f_1 \sqcup f_2$. In order to combine constructions on such local decompositions back into a global construction, we need to understand the symmetry group $\A_f $.}
\label{tubedmap}
\end{figure}

The fibered family of singularities in fig.~\ref{tubedmap} defines a fiber bundle, whose structure group is the stabilizer $\A_f$ of $\A$-equivalences leaving $f$ invariant. In order to patch local constructions on such local fiber bundles together, it becomes essential to understand $\A_f$. In reality, the germ $f$ will often be a multigerm, where each point in the target of the map corresponds to several source points, just like in fig.~\ref{tubedmap}. Thus, we need to understand $\A_f$ in the case where $f$ has several source basepoints. For instance, knowledge of the structure of $\A_f$ for multigerms $f$ is used in \cite{feragen,feragenthesis} to prove topological stability of maps.

In this article, we show that given a finitely $\A$-determined multigerm $f$, the group $\A_f$ has a maximal compact subgroup, which is unique up to conjugacy. We denote an arbitrary representative of the conjugacy class by $MC(\A_f)$, and we show that the quotient $\A_f/MC(\A_f)$ is contractible, where contractibility is defined through a smooth map extension property. Geometrically, this means that we can reduce the non-topological structure group $\A_f$ of the tubular neighborhood bundle to the compact, finite-dimensional Lie group given by $MC(\A_f)$. 

Most of the results presented in this article have been proven for monogerms by K.~J\"anich \cite{janich}, C.T.C.~Wall \cite{wallsymm} and R.~Rim{\'a}nyi \cite{rimanyi}. Other groups of germ equivalences have also been studied. In his thesis \cite[Theorem~1.6.3]{rimanyithesis}, Rim{\'a}nyi states our main theorem in the case of $\A$-equivalences for stable multigerms, without proof. 

For some classical groups of germ equivalences, the step from monogerms to multigerms is trivial. It is easy to see that we can decompose the groups $\K$ and $\K_f$, defined on p.~\pageref{kdef}, for multigerms $f = f_1 \sqcup \ldots \sqcup f_s \colon \bigsqcup_s (\R^n,0) \to (\R^p,0)$ to a product of $\K$-groups for monogerms: $\K(s-\textrm{multigerms}) = \K^s$ and $\K_f = \K_{f_1} \times \ldots \times \K_{f_s}$. \label{unionK} Similarly, for the group of source diffeomorphisms, or $\mathscr{R}$-\emph{equivalences}, we have $\mathscr{R}(s-\textrm{multigerms}) = \mathscr{R}^s$ and $\mathscr{R}_f = \mathscr{R}_{f_1} \times \ldots \times \mathscr{R}_{f_s}$. However, in the case of $\A$-equivalences, the multigerm case does not in any obvious way reduce to the monogerm case.

By lemma~\ref{standardstablemultigerm}, any stable multigerm $F \colon (\R^n, S) \to (\R^p,0)$ with $|S|=s<\infty$ admits a decomposition
\begin{equation} \label{multigermdecomp}
\left( \bigsqcup_{i=1}^s \sigma_i \circ \left( f_i \times \id_{\R^{p-p_i-d}} \right) \right) \times \id_{\R^d} \colon \left( \bigsqcup_{i=1}^s \R^{n_i} \times \prod_{j \neq i} \R^{p_j} \right) \times \R^d \to \left( \prod_{i=1}^s \R^{p_i} \right) \times \R^d
\end{equation}
in suitably chosen coordinates, where the $f_i$ are ministable germs which are $\EK$-equivalent to the germs of $F$ at points in $S$ (see \cite[p.~30]{pw}). If $F$ is ministable, then $d$ will be $0$. (In Rim{\'a}nyi's terminology, the $f_i$ will be \emph{roots of their kinds}.)

There is a natural embedding of the product $\A_{f_1} \times \ldots \times \A_{f_s}$ into $\A_F$, given by
\begin{equation} \label{aemb}
\iota \colon \left( (\phi_1, \psi_1), \ldots, (\phi_s, \psi_s)\right) \mapsto \left( \bigsqcup_{i=1}^s \sigma_i \circ (\phi_i \times \prod_{j=1\\ j \neq i}^s \psi_j) , (\psi_1, \ldots, \psi_s) \right),
\end{equation}
where $\sigma_i$ moves the first component to the $i^{\textrm{th}}$. However, there is no guarantee that this embedding is surjective. Suppose that there are two ways of decomposing $F$ as in (\ref{multigermdecomp}), and suppose for simplicity that $d=0$. That is, there exist some diffeomorphism germs $\Phi_i$, $i=1, \ldots, s$ and $\Psi$, such that the following commutes:

\[
\begin{diagram}
\dgARROWLENGTH=6em
\node{\bigsqcup_{i=1}^s \R^{n_i} \times \prod_{j \neq i} \R^{p_j}} \arrow{e,t}{\bigsqcup_{i=1}^s \sigma_i \circ (f_i \times \id)}
\arrow{s,l}{\bigsqcup_{i=1}^s \Phi_i} \node{\prod_{i=1}^s \R^{p_i}} \arrow{s,r}{\Psi}\\
\node{\bigsqcup_{i=1}^s \R^{n_i} \prod_{j \neq i} \R^{p_j}} \arrow{e,b}{\bigsqcup_{i=1}^s \sigma_i \circ (f_i \times \id)} \node{\prod_{i=1}^s \R^{p_i}}
\end{diagram}
\]

Given an element $\iota\left( (\phi_1, \psi_1), \ldots, (\phi_s, \psi_s) \right) \in \A_F$, there is no obvious guarantee that
\[
\alpha = \left( \bigsqcup_{i=1}^s \Phi_i \circ \left( \sigma_i \circ (\phi_i \times \prod_{j \neq i} \psi_j) \right), \Psi \circ (\psi_1, \ldots, \psi_s) \right)
\]
is of the form on the right hand side of (\ref{aemb}). What is easy to show is that the axes
\[
\sigma_i \left( \R^{n_i} \times \{0\} \right) \subset \sigma_i \left(\R^{n_i} \times \prod \R^{p_j} \right),
\]
are left invariant under $\alpha$, since this is where the singularities $f_i$ are found, but this is not enough.

We had hoped that it might be easier to find a decomposition for compact subgroups $G < \A_{f}$, as these are conjugate to linear subgroups of $\A$. However, we meet yet another problem -- can we find coordinates that \emph{simultaneously} give $f$ in the form (\ref{multigermdecomp}) \emph{and} linearize $G$? This problem might be solved by replacing $G$ by the $1$-jet of $G$, which is isomorphic to $G$ -- but $j^1G$ would not generally leave $f$ invariant, and hence would not be a subgroup of $\A_f$.

We shall, indeed, see that for statements concerning maximal compact subgroups of $\A_f$, we \emph{can} reduce to the monogerm case -- but this is not trivial. For statements concerning all of $\A_f$ -- in particular concerning the contractibility of the quotient $\A_f/MC(\A_f)$ -- we need to reprove the theorems for multigerms. We shall also investigate a decomposition theorem for $\A_f$ in terms of $\mathscr{R}_f$ and a subgroup of $\mathscr{L}$. Finally, we show that for a large class of multigerms $f$, the maximal compact subgroups $MC(\A_f)$ are very simple.

Some of the results presented in this article can be found in A.~Feragen's PhD thesis \cite{feragenthesis}, to which we refer for a higher level of detail.

\section{Preliminaries and terminology} \label{preliminaries}

We denote by $\mathscr{E}(n)$ the ring of function germs $(\R^n,0) \to \R$, and denote by $m(n)$ the ideal in $\mathscr{E}(n)$ consisting of germs that map $0$ to $0$. We write $\mathscr{E}(n,p)$ for the set of germs $(\R^n,0) \to \R^p$, and note that $\mathscr{E}(n,p) \cong \bigoplus_p \mathscr{E}(n)$.

Let $f \colon (M,S) \to (N,y)$ be a smooth germ. Define the set $\theta_f$ of \emph{vector fields along $f$} to be the set 
\[
\{ \theta \colon N \to TP | \pi_P \circ \theta = f\},
\]
and set
\[
\begin{array}{l}
\theta_S=\theta_{(N,S)}=\theta_{\id_{(N,S)}}, \textrm{ and}\\
\theta_y=\theta_{(P,y)}=\theta_{\id_{(P,y)}},
\end{array}
\]
where we choose notation depending on how explicit we need to be. We define $tf \colon \theta_S \to \theta_f$ and $wf \colon \theta_y \to \theta_f$ by setting \label{infstabpageref}
\[
\begin{array}{l}
tf(\xi)=Tf \circ \xi,\\
wf(\eta)=\eta \circ f.
\end{array} 
\]
In this article, $S$ will always be a finite set $\{x_1, \ldots, x_s\}$, and we can, up to a choice of local coordinates, always write a germ $f \colon (N,S) \to (P,y)$ in the form

\begin{equation} \label{multigermnormalform1}
f:=f_1 \sqcup f_2 \sqcup \ldots \sqcup f_s \colon \underbrace{(\R^n,0) \sqcup \ldots \sqcup (\R^n,0)}_{s \textrm{ copies}} \to (\R^p,0),
\end{equation}

where $n$ and $p$ are the dimensions of $N$ and $P$, respectively.

We shall need the following standard choice of coordinates for stable multigerms:

\begin{lem} \label{standardstablemultigerm}
Let $f \colon \bigsqcup_{i = 1}^s (\R^n, 0) \to (\R^p,0)$ be a stable multigerm. By choosing suitable local coordinates, we may assume that $f$ is of the form
\[
F \colon 
\bigsqcup_{i=1}^s \left( \R^{n_i} \times \prod_{j \neq i} \R^{p_j} \times \R^d, 0 \right) \to \left( \prod_{j=1}^s \R^{p_j} \times \R^d, 0 \right),
\]
where $F=\bigsqcup_{i=1}^s \sigma_i \circ (F_i \times \emph{\id}_{\prod_{j \neq i} t(F_j) \times \R^d})$; $F_i \colon \R^{n_i} \to \R^{p_i}$ is ministable; the local algebra of $F_i$ is isomorphic to the local algebra of the germ of $f$ at $x_i$; $\sigma_i$ moves the $1^{st}$ coordinate to the $i^{\textrm{th}}$ in $\prod_{i=1}^s t(F_i)$; and $d \in \N_0$ is chosen to get the appropriate dimensions. Here, $s(F_j)$ denotes the source of $F_j$, and $t(F_j)$ denotes the target of $F_j$.
\end{lem}

\begin{proof}
Since the $F_i$ are unfoldings of the $f_i$, the local algebras $Q(F_i)$ and $Q(f_i)$ are isomorphic, and as a consequence also the local algebras $Q(F)$ and $Q(f)$ are isomorphic \cite[Theorem~2.1]{ma4}. Hence $F$ and $f$ are $\K$-equivalent. But $F$ and $f$ are also both stable, hence $F$ and $f$ are $\A$-equivalent \cite[Chapter III Theorem 4.3]{gib}.\qed
\end{proof}

\section{Groups of multigerm equivalences}

Suppose given a multigerm $f = f_1 \sqcup f_2 \sqcup \ldots \sqcup f_s$ as in (\ref{multigermnormalform1}). Recall that if $\mathscr{R}$ is the group of diffeomorphism germs of a source component $(\R^n,0)$, and $\mathscr{L}$ is the group of diffeomorphism germs of the target $(\R^p,0)$, then $\A$ denotes the group of equivalences on the space of multigerms
\begin{equation} \label{germspace}
C^\infty \left((\R^n,0) \sqcup \ldots \sqcup (\R^n,0), (\R^p,0)\right)
\end{equation}
defined by $\bigsqcup_s \mathscr{R} \sqcup \mathscr{L}$. We will write elements of $\A$ either in the form $\phi=(\phi_1, \ldots, \phi_s, \phi_t)$ where the $\phi_i$ are the diffeomorphisms in question, or sometimes in the form $\psi = (\psi_1, \ldots, \psi_s)$ where $\psi_i=\phi_i \times \phi_t$ for each $i$. We use the latter notation when we consider $\psi$ as an element in the group $\K$, as defined below. Note that with this definition, the $\A$-equivalences leave the source base points fixed.

The group $\K$ of equivalences on the space of multigerms (\ref{germspace}) is the group of diffeomorphism germs \label{kdef}
\[
H \colon \left(\bigsqcup_s \R^n \times \R^p, \bigsqcup_s (0,0)\right) \to \left(\bigsqcup_s \R^n \times \R^p, \bigsqcup_s (0,0)\right)
\] 
such that the following diagram commutes:
\[
\begin{diagram}
\node{\bigsqcup_s (\R^n, 0)} \arrow{e,t}{\id \times 0} \arrow{s,l}{H_0} \node{\bigsqcup_s \left(\R^n \times \R^p, (0,0)\right)} \arrow{e,t}{\pr_{\R^n}} \arrow{s,l}{H} \node{\bigsqcup_s (\R^n,0)} \arrow{s,r}{H_0}\\
\node{\bigsqcup_s (\R^n,0)} \arrow{e,b}{\id \times 0} \node{\bigsqcup_s \left(\R^n \times \R^p, (0,0)\right)} \arrow{e,b}{\pr_{\R^n}} \node{\bigsqcup_s (\R^n,0)}
\end{diagram}
\]
where $H_0 = H|\R^n \times \{0\}$. We will write elements of $\K$ in the form $\psi=(\psi_1, \ldots, \psi_s)$ where, in fact, $\psi=\bigsqcup_i \psi_i$. The group $\K$ acts on germs $f$ in the following way: $H \cdot f = g$ if $(\id,f) \circ H_0=H \circ (\id, g)$.

Given a multigerm $f$, we denote by $\A_f$ or $\K_f$ the \emph{stabilizer} of $\A$ or $\K$ at $f$; namely 
\[
\HH_f=\{\psi \in \HH : \psi \cdot f = f\}
\]
for $\HH=\A$ or $\K$.

Following the original definition by J\"{a}nich~\cite{janich}, we define a compact subgroup of $\A$ to be a subgroup $G$ of $\A$ which is conjugate in $\A$ to a compact subgroup of $\bigsqcup_s GL_n \sqcup GL_p < \A$. The definition is reasonable: Suppose that $G < \A$ is isomorphic to some compact Lie group $\tilde{G}$, such that $\tilde{G}$ acts diffeomorphically on $\bigsqcup_s \R^n \sqcup \R^p$ through the isomorphism $\tilde{G} \to G$, keeping the origins fixed (compare with the definition of \cite{dpwil}). Then, in particular, $\tilde{G}$ acts diffeomorphically on each of the $\R^n$ and on $\R^p$. By Bochner's theorem we can choose local coordinates in $(\R^n,0), \ldots, (\R^n, 0)$ and $(\R^p,0)$, respectively, given by diffeomorphism germs $\phi_1, \ldots, \phi_s$ and $\phi_t$, with respect to which $G$ acts linearly on the $\R^n$ and on $\R^p$. These define an element $(\phi_1, \ldots, \phi_s, \phi_t)$ of $\A$, which linearizes $G$ in $\A$.

By analogy, we define a compact subgroup of $\K$ to be a subgroup $G < \K$ which is conjugate in $\K$ to a compact linear subgroup of $\bigsqcup_s GL_n \times GL_p$. Throughout the rest of the article, we let $\HH$ denote $\A$ or $\K$ unless otherwise is specified.

Given a compact subgroup $G$ of a group $\HH$ of diffeomorphism-germs such as $\A$, $\K$ or $\mathscr{R}$, we say that $G$ is a \emph{maximal compact subgroup} of $\HH$ if any other compact subgroup $H$ of $\HH$ is conjugate in $\HH$ to a subgroup of $G$.

We start out by making a few basic observations concerning the groups of $\A$-equivalences and $\K$-equivalences:

\begin{lem} \label{basicobservations}
\begin{itemize}
\item[i)] We have $\A_f < \K_f$, so in particular $MC(\A_f) < MC(\K_f)$, assuming the maximal compact subgroups exist.
\item[ii)] Suppose that $f \colon (\R^n,0) \to (\R^p,0)$ is a monogerm, and consider $\A$ and $\K$ as groups acting on the source and target of $f$. Then $\K \cap GL_{n+p}$ is a subgroup of $\A$, so in particular, $\K_f \cap GL_{n+p} < \A_f$.
\item[iii)] Suppose that a multigerm $f \colon \bigsqcup_s (\R^n,0) \to (\R^p,0)$ is finitely $k-\HH$-determined, and suppose that $G < \HH_f$ is compact. Then we can find an element $\phi \in \HH$ and a subgroup $\tilde{G}$ of $\HH_{\phi \cdot f}$ such that $\phi \cdot f$ is a polynomial of degree $\le k$, and $\tilde{G}$ is a linear group which is conjugate (in $\HH$) to $G$.
\item[iv)] For any compact subgroup $G$ of $\HH$ the restriction of the $1$-jet map $j^1| \colon G \to \bigsqcup_s GL_n \times GL_p$ is injective.
\end{itemize} 
\end{lem}

\begin{proof}
Claim \textit{i)} is trivial. 
\begin{itemize}
\item[\textit{ii)}] An element $H \in \K$ is an element of $\A$ if it can be decomposed as $H = H_1 \times H_2$, where $H_1 \colon (\R^n,0) \to (\R^n,0)$ and $H_2 \colon (\R^p,0) \to (\R^p,0)$. Since $H \in \K$, we know that $H(\R^n,0) \subset (\R^n,0)$, so we set $H_1 = H | (\R^n,0) \times \{0\}$. Since $H$ is a linear diffeomorphism, $H(a) \in H_1$ for some $a \in (\R^{n+p},0)$ implies $a \in H_1$. Then we must have $H(b) \in H_1^\bot = \{0\} \times (\R^p,0):=H_2$, for all $b \in \{0\} \times (\R^p,0)$, since otherwise, we would find $a = b - H^{-1}(\pr_{H_2}(H(b))) \notin H_1$ with $H(a) \in H_1$.

\item[\textit{iii)}] Since $G < \HH_f$ is compact, we can find $\psi \in \HH$ such that $\psi G \psi^{-1}$ is linear, by definition. Now $\tilde{G}=\psi G \psi^{-1} < \HH_{\psi \cdot f}$, and since $f$ is $k-\HH$-determined, so is $\psi \cdot f$. In particular, $\psi \cdot f$ is $\HH$-equivalent to the polynomial representative $p$ of its $k$-jet $j^k (\psi \cdot f)$, by an element $\tilde{\psi} \in \HH$, say:
\[
p=\tilde{\psi} \cdot \psi \cdot f.
\]
We claim that $\tilde{G} < \HH_p$. Since the linear group $\tilde{G}$ leaves $\psi \cdot f$ invariant, it certainly leaves the jet $j^k(\psi \cdot f)$ invariant in $J^k(n,p)$. But then it must leave $p$ invariant, since then, for any $g \in \tilde{G}$, the map $g \cdot p$ is a degree $k$ polynomial representing $j^k(\psi \cdot f)$. Set $\phi=\tilde{\psi} \cdot \psi$, and we are done.

\item[\textit{iv)}] By the definition of a compact subgroup, there exists a choice of coordinates on $\bigsqcup_s \R^n \times \R^p$ such that $G$ acts linearly; now in these coordinates the $1$-jet map is just the inclusion into $\bigsqcup_s GL_n \times GL_p$. The topological properties of the map $j^1$ do not depend on the choice of coordinates; hence $j^1|G$ is injective.  \qed
\end{itemize}
\end{proof}

\subsection{Maximal compact subgroups}

Now we are ready to state and prove the main theorem of the section.

\begin{thm} \label{existenceofmc}
Let $f$ be a finitely $\HH$-determined multigerm as in (\ref{multigermnormalform1}). The group $\HH_f$ has a maximal compact subgroup, which is unique up to conjugation in $\HH_f$.
\end{thm}

The monogerm version of this theorem was proven by J\"anich \cite{janich} (for $\HH=\mathscr{R}$) and Wall \cite{wallsymm} (for $\HH=\A$ or $\K$) with some completing comments by du Plessis and Wilson \cite[p.~270]{dpwil}, who proved similar results for actions of $\mathscr{R}$, but without finite $\mathscr{R}$-determinacy.

\begin{proof}
Just like in the monogerm case we will need an equivariant finite determinacy condition, as formulated in the lemma below. The proof of the lemma follows the standard proof of $k$-determinacy \cite{ma3}, using vector fields -- but averaging the vector fields over the group $G$ using the Haar integral.

\begin{lem} \label{equivdeterminacy}
Let $f \colon \bigsqcup_s (\R^n, 0) \to (\R^p,0)$ be $k-\HH$-determined, and suppose that $G$ is a compact linear subgroup of $\HH_f$. Then $f$ is $k-G$-determined; that is, if $\tilde{f} \colon \bigsqcup_s (\R^n,0) \to (\R^p,0)$ is $G$-invariant such that $j^k f=j^k \tilde{f}$, then there exists $\phi \in \HH$ such that $\phi \cdot g=g \cdot \phi$ for all $g \in G$, and $\tilde{f}=\phi \cdot f$. In particular, $G < \HH_{\tilde{f}}$, since $g \cdot \tilde{f} = g \cdot \phi \cdot f = \phi \cdot g \cdot f = \phi \cdot f = \tilde{f}$ for all $g \in G$. \qed
\end{lem}

Assume that $f$ is $k-\HH$-determined. We write $\HH^k$ for the families of invertible $k$-jets
\[
\phi^k = (\phi_1^k, \ldots, \phi_s^k) \colon \left( \bigsqcup_ s \R^n \times \R^p, \bigsqcup_s 0 \right) \to \left( \bigsqcup_s \R^n \times \R^p, \bigsqcup_s 0 \right), \quad \phi \in \HH,
\]
and write $\HH^k_f$ for the subgroup stabilizing $j^k f$; both these groups are real algebraic.

Real algebraic groups have finitely many components, so we may apply Iwasawa's theorem \cite[p.~180]{hoch} and choose a maximal compact subgroup $G$ of $\HH^k_f$. Following the arguments of J\"anich \cite[\S 2]{janich} and Bochner \cite{bochner} we can linearize $G$ in the following ways:

\begin{lem}  \label{stronglin}
\begin{itemize}
\item[i)] Suppose that $G < \HH^k$ is a compact subgroup. Then there exists $\phi \in \HH^k$ such that $\phi G \phi^{-1} < \bigsqcup_s GL_n \times GL_p$.
\item[ii)] Suppose that $G < \HH$ is a compact subgroup. Then there exists $\phi \in \HH$ such that $\phi G\phi^{-1} < \bigsqcup_s GL_n \times GL_p$. Suppose, furthermore, that $j^kG < \bigsqcup_s GL_n \times GL_p < \HH^k$. Then we may assume that $j^k \phi=(1, \ldots, 1)$. \qed
\end{itemize}
\end{lem}

We return to the proof of theorem~\ref{existenceofmc}, and to the maximal compact subgroup $G$ of $\HH^k_f$.

Denote $G_0=\phi G\phi^{-1} < \bigsqcup_s GL_n \times GL_p$. By abuse of notation, we will also denote by $G_0$ the corresponding linear subgroup of $\HH^k$. Let $\tilde{\phi} \in \HH$ such that $j^k \tilde{\phi}=\phi$. Then if $f_0=\tilde{\phi} \cdot f$, its jet $j^kf_0$ is $G_0$-invariant, and $G_0$ is maximal compact in $\HH^k_{f_0}$.

Let $H$ be any compact subgroup of $\HH_{f_0}$. By classical Lie group theory $j^k H$ is conjugate in $\HH^k_{f_0}$ to a subgroup of $G_0$, say by a family of jets of diffeomorphisms $\psi^k=(\psi_1^k, \ldots, \psi^k_s)$. Let $\tilde{\psi}$ be a family of diffeomorphisms $(\tilde{\psi}_1, \ldots, \tilde{\psi}_s)$ with jet $\psi^k$. Then the $k$-jet of $\tilde{\psi}H\tilde{\psi}^{-1}$ is linear. 

By lemma~\ref{stronglin}, we can find $\psi=(\psi_1, \ldots, \psi_s)$ with the same $k$-jet as $\tilde{\psi}$, such that $\psi H \psi^{-1}$ is linear. Hence $\tilde{f}=\psi \cdot f_0$ is $(\psi H \psi^{-1})$-invariant, just like $f_0$ (because $\psi H \psi^{-1} < G_0$). Since $j^k \psi=j^k \tilde{\psi}=\psi^k \in \HH^k_{f_0}$, we have $j^k \tilde{f}=j^k f_0$. Recall that $f_0$ is $k-(\psi H \psi^{-1})$-determined, so by lemma~\ref{equivdeterminacy} there exists some $(\psi H \psi^{-1})$-equivariant family of diffeomorphisms $\alpha=(\alpha_1, \ldots, \alpha_s) \in \HH$ such that $f_0=\alpha \cdot \tilde{f}$. Hence $f_0=(\alpha \psi) \cdot f_0$, and so $\alpha \psi$ preserves $f_0$ \emph{and} conjugates $H$ to
\[
\alpha \underbrace{\psi H \psi^{-1}}_{< G_0} \alpha^{-1} = \psi H \psi^{-1} < G_0
\]
where the last equality holds because $\alpha$ is $G_0$-equivariant.

We have seen that $G_0$ can be viewed as a compact subgroup of $\HH_{f_0}$, and that any other compact subgroup $H$ of $\HH_{f_0}$ is conjugate to a subgroup of $G_0$, so conjugating back to $\HH_f$, we see that theorem~\ref{existenceofmc} holds.  \qed
\end{proof}

\subsection{Contractibility of quotients}

Following J\"anich \cite{janich}, we define what it means for $\A_f/MC(\A_f)$ to be contractible. We have not specified a topology on $\A_f/MC(\A_f)$, and in fact we shall define contractibility not in terms of the topology of $\A_f/MC(\A_f)$ as a space of its own, but through the topological properties of its action on the source and target of $f$.

Before defining contractibility, we define what it means for a map into a quotient $\A_f/G$ to be smooth. Let $G$ be a subgroup of $\A$. Given a smooth manifold $M$, possibly with boundary, we say that a map $q \colon M \to \A/G$ is \emph{smooth} if there exists an open covering $\{U_i\}$ of $M$ such that $q$ is represented by fibered (over $U_i$) maps $\phi_i \colon U_i \times \bigsqcup_s \R^n \to U_i \times \bigsqcup_s \R^n$ and $\psi_i \colon U_i \times \R^p \to U_i \times \R^p$, which are diffeomorphisms. 

Equivalently, a map $\alpha \colon M \to \A/G$ is smooth if there exists an open covering $\{U_i\}_{i \in I}$ of $M$ such that $\alpha$ admits a local lift $\tilde{\alpha}_i \colon U_i \to \A$, where the corresponding fibered map-germs $\phi_i$ and $\psi_i$ are smooth.

\begin{defn}
Let $G$ be a subgroup of $\A_f$. The quotient $\A_f/G$ is \emph{contractible} if for every smooth manifold $M$ with boundary, any smooth map $\partial M \to \A_f/G$ can be extended to a smooth map $M \to \A_f/G$. 
\end{defn}

We proceed to state and prove the main result of the section:

\begin{thm} \label{contractibilitythm}
Suppose given a finitely $\A$-determined multigerm
\[
f = \bigsqcup_{i=1}^s f_i \colon \bigsqcup_{i=1}^s (\R^n,0) \to (\R^p,0).
\]
Then the quotient $\A_f/MC(\A_f)$ is contractible. 
\end{thm}

This proof follows that of Rim{\'a}nyi \cite{rimanyithesis} for the monogerm case.

\begin{proof}
The following proposition is crucial to the proof:

\begin{prop} \label{extensionprop}
There exists an $l \in \N$ such that the following holds:

If $M$ is an $r$-dimensional manifold with boundary (possibly empty) and
\[
g,h \colon (M \times \bigsqcup_s \R^n, M \times \bigsqcup_s \{0\}) \to (M \times \R^p, M \times \{0\})
\]
are fibered (over $M$) germs at $M \times \bigsqcup_s \{0\}$ satisfying the following properties:
\[
g|\partial M \times \bigsqcup_s (\R^n,0) = h|\partial M \times \bigsqcup_s (\R^n,0), \textrm{ and }
\]
\[
j^l(g|u \times \R^n)=j^l(h|u \times \R^n)=j^l f \quad \ \forall u \in M,
\]
then there exist $\psi^k \in \textrm{\emph{Diff}}(M \times \R^n)$, $k=1, \ldots, s$, and $\phi \in \textrm{\emph{Diff}}(M \times \R^p)$ such that $g=\phi \circ h \circ \bigsqcup_s (\psi^k)^{-1}$ and
\[
\psi^k|\partial M \times \R^n=\emph{\id}, \qquad \phi|\partial M \times \R^p = \emph{\id},
\]
\[
j^1(\psi^k|u \times \R^n)=\emph{\id}, \qquad j^1(\phi|u \times \R^p)=\emph{\id},
\]
for all $u \in M$, $k=1, \ldots, s$.
\end{prop}

\begin{rem}
Paraphrased, proposition~\ref{extensionprop} says: Given a pair of ''smooth'' maps $g, h \colon M \to C^\infty\left((\R^n,0), (\R^p,0)\right)$ such that $g|\partial M = h|\partial M$ and $j^l(g(u)) = j^l(h(u)) = j^l f$ for all $u \in M$, there exists a smooth map $\varphi \colon M \to \A$ such that $\varphi | \partial M \equiv \id$, $j^1 \varphi \equiv \id$ and $g = \varphi \cdot h$. 
\end{rem}

\begin{proof}
We find the $\psi^k$ and $\phi$ by using the flows of suitably chosen vector fields in the source components and in target.

Let $F \colon \left( M \times \bigsqcup_s \R^n \times \R, M \times \bigsqcup_s \{0\} \times \R \right) \to \left( M \times \R^p \times \R, M \times \{0\} \times \R \right)$ be the map germ defined by $(u,x,t) \mapsto \left((1-t)g(u,x)+th(u,x), t\right)$. From now on, we denote by $u=(u_i), \quad x^k=(x^k_i), \quad y=(y_i), \quad t$, the coordinates of $M$, the $k^\textrm{\textrm{th}}$ source component $\R^n$, $\R^p$, and $\R$, respectively. We write $F=\bigsqcup_{k=1}^s F^k$, and the notation $F_y$ will denote the composition $\pr_{\R^p} \circ F$, and so on.

\subsubsection*{Constructing the diffeomorphisms}

We want to construct the diffeomorphisms $\psi^k$ and $\phi$ through flows $\Psi$ and $\Phi$ in the source and target of $F$ such that
\[
\begin{array}{l}
\Psi| \colon (M \times \bigsqcup_s\R^n \times \R) \times [0,1] \to M \times \bigsqcup_s \R^n \times \R, \quad \Psi=\bigsqcup_{k=1}^s \Psi^k,\\
\Phi| \colon (M \times \R^p \times \R) \times [0,1] \to M \times \R^p \times \R,
\end{array}
\]
with
\[
\begin{array}{l}
\Psi^k \left( (u,x^k,0),s \right) \in M \times \R^n \times \{s\},\\
\Phi \left( (u,y,0),s \right) \in M \times \R^p \times \{s\},
\end{array}
\]for all $s \in [0,1]$, and $F(\Psi((u,x,0),s))=\Phi((g(u,x),0),s)=\Phi(F(u,x,0),s)$, which holds if
\begin{equation} \label{cond-1}
F \circ \Psi = \Phi \circ (F \times \pr_{[0,1]}).
\end{equation}

Suppose that we have found such flows $\Psi$ and $\Phi$, and define maps 
\[
\begin{array}{l}
\tilde{\Psi} \colon M \times \bigsqcup_s \R^n \times [0,1] \to M \times \bigsqcup_s \R^n,\\
\tilde{\Phi} \colon M \times \R^p \times[0,1] \to M \times \R^p,
\end{array}
\]
by setting
\[
\begin{array}{l}
\tilde{\Psi}(u,x,s)=\pr_{M \times \bigsqcup_s \R^n} \circ \Psi((u,x,0),s),\\
\tilde{\Phi}(u,y,s)=\pr_{M \times \R^p} \circ \Phi((u,y,0),s),
\end{array}
\]
and define $\tilde{h} \colon M \times \bigsqcup_s \R^n \times [0,1] \to M \times \R^p$ by setting 
\[
\tilde{h}(u,x,s)=\left( \tilde{\Phi}_s^{-1} \circ h \circ \tilde{\Psi} \right)(u,x,s), 
\]
where $\tilde{\Phi}_s(u,y)=\tilde{\Phi}(u,y,s)$. Note that $\tilde{\Phi}_0(u,y)=(u,y)$, and that $\tilde{\Psi}(u,x,0)=(u,x)$.

\begin{lem}
Then $\tilde{h}_0=h$ and $\tilde{h}_1=g$.
\end{lem}

\begin{proof}
It is straightforward to prove that $\tilde{h}_0=h$, and for the second identity we note that $\tilde{h}_1(u,x)=\tilde{h}(u,x,1)=\tilde{\Phi}_1^{-1}(h(\tilde{\Psi}(u,x,1)))$, so $\tilde{h}_1=g$ if and only if $(\tilde{\Phi}_1 \circ g)(u,x)=(h \circ \tilde{\Psi})(u,x,1)$ for all $u,x$, which holds if and only if
\begin{equation} \label{floweq}
\tilde{\Phi}(g(u,x),1)=h(\tilde{\Psi}(u,x,1)) \textrm{ for all } u,x,
\end{equation}
where $h(\tilde{\Psi}(u,x,1))=h(\pr_{M \times \bigsqcup_s \R^n}(\Psi((u,x,0),1)))$. But we have  
\[
F\left(\Psi\left((u,x,0),1\right)\right)=\Phi\left(F(u,x,0),1\right)
\]
by (\ref{cond-1}), and $\Phi\left(F(u,x,0),1\right)=\Phi(g(u,x),0,1)$ by the definition of $F$, while $F\left(\Psi\left((u,x,0),1\right)\right)=\left(h\left(\pr_{M \times \bigsqcup_s \R^n}(\Psi((u,x,0),1))\right), 1\right)$, also by the definition of $F$, so $h(\tilde{\Psi}(u,x,1)) = \pr_{M \times \R^p}(F(\Psi((u,x,0),1))) = \pr_{M \times \R^p}(\Phi(g(u,x),0,1)) = \tilde{\Phi}(g(u,x),1),$ and (\ref{floweq}) holds. \qed
\end{proof}

In particular, $\tilde{\Phi}_1^{-1} \circ h \circ \tilde{\Psi}_1=\tilde{h}_1=g$, and thus $\tilde{\Psi}_1$ and $\tilde{\Phi}_1$ are the conjugating diffeomorphisms sought $\bigsqcup \psi^k$ and $\phi$.

Suppose that we are given map germs
\[
\begin{array}{l}
X^k \colon (M \times \R^n \times \R, M \times 0 \times \R) \to (\R^n,0) \ (k=1, \ldots, s),\\
Y \colon (M \times \R^p \times \R, M \times 0 \times \R) \to (\R^p,0),
\end{array}
\]
such that the following conditions (\ref{cond1}) -- (\ref{cond4}) hold:
\begin{equation} \label{cond1}
\sum_{i=1}^n \frac{\partial F_{y_j}}{\partial x_i}(u,x^k,t)X^k_i(u, x^k, t) + \frac{\partial F_{y_j}}{\partial t}(u,x^k,t) = Y_j(F(u, x^k, t))
\end{equation}
for all $j=1, \ldots, p$, and $k=1, \ldots, s$;
\begin{equation} \label{cond2}
X^k|M \times 0 \times \R=0 \ (k=1, \ldots, s); \qquad Y|M \times 0 \times \R=0;
\end{equation}
\begin{equation} \label{cond3}
\left\{\begin{array}{l}
\frac{\partial X^k}{\partial x_i}(u,0,t)=0 \textrm{ for all } i=1, \ldots, n, \quad (k=1, \ldots, s); \\ 
 \\
\frac{\partial Y}{\partial y_j}(u,0,t)=0 \textrm{ for all } j=1, \ldots, p;
\end{array} \right.
\end{equation}
\begin{equation} \label{cond4}
X^k|\partial M \times \R^n \times \R = 0; \qquad Y|\partial M \times \R^p \times \R = 0.
\end{equation}

Consider the flows of the following vector fields:
\[
\begin{array}{ll}
\tilde{X}^k \colon M \times \R^n \times \R \to TM \times T\R^n \times T\R, & (u, x^k, t) \mapsto (0, X(u, x^k, t), 1),\\
\tilde{Y} \colon M \times \R^p \times \R \mapsto TM \times T\R^p \times T\R, & (u,y,t) \mapsto (0, Y(u,y,t), 1).
\end{array}
\]

By (\ref{cond2}), these flows exist at least in a neighborhood of $M \times 0 \times \R$. Furthermore, we see that the condition (\ref{cond1}) is just the condition of being a derivative of $F$-related flows, that is, a pair of flows satisfying (\ref{cond-1}). 

The maps $\tilde{\Psi}_1=\bigsqcup_{k=1}^s \tilde{\Psi}^k_1$ and $\tilde{\Phi}_1$ associated with the flows of $\tilde{X}^k$ and $\tilde{Y}$ as described above clearly satisfy
\[
\begin{array}{ll}
\tilde{\Psi}_1^k| \partial M \times \R^n = \id & \tilde{\Phi}_1|\partial M \times \R^p = \id,\\
j^1(\tilde{\Psi}_1^k|u \times \R^n)=\id & j^1(\tilde{\Phi}|u \times \R^p)=\id,
\end{array}
\]
by (\ref{cond3}) and (\ref{cond4}). Hence, the proof of proposition~\ref{extensionprop} is completed by finding the vector fields $X^k$ and $Y$. This construction is quite long and technical, but follows the construction by Rimanyi for the monogerm case \cite{rimanyi}. We refer to the thesis~\cite{feragenthesis} for details. \qed
\end{proof}

We may now return to the proof of theorem~\ref{contractibilitythm}.

Denote by $\A^l$ the Lie group of $l$-jets of elements of $\A$ and set 
\[
\A_f^l=\{(z_1, \ldots, z_s, z_t) \in \A^l|z_t \circ j^lf \circ (z_1 \sqcup \ldots \sqcup z_s)^{-1}=j^l f \}.
\]

For a sufficiently large $l \in \N$, the image of a maximal compact subgroup of $\A_f$ under $j^l$ is a maximal compact subgroup of $\A_f^l$. This is a consequence of the apparently more general:

\begin{lem}
For $f$ finitely $\A$-determined and $l$ sufficiently large, any maximal compact subgroup $G$ of $\A^l_f$ is the image under $j^l$ of a maximal compact subgroup $\tilde{G}$ of $\A_f$. 
\end{lem}

\begin{proof}
Let $G < \A^l_f$ be a maximal compact subgroup. Up to conjugation by some element $\phi^l \in \A^l$, $G$ is linear, i.e.,~$\phi^l G (\phi^l)^{-1} \equiv G_\phi$ is a linear maximal subgroup of $(\A^l)_{\phi^l \cdot f^l} = \A^l_{\phi \cdot f}$. But then $G_\phi$ (as a matrix group) is a subgroup of $\A_{p(\phi \cdot f)}$, where $p(\phi \cdot f)$ is the $l^{\textrm{th}}$ degree polynomial map representing $\phi \cdot f$. For $l$ sufficiently large, $p(\phi \cdot f)$ is $\A$-equivalent to $\phi \cdot f$ via some element $\psi \in \A$: $\phi \cdot f = \psi \cdot p(\phi \cdot f)$.

We must have $\psi^l :=j^l \psi \in \A^l_{\phi \cdot f}$, since $j^l(\phi \cdot f) = j^l \psi \cdot j^l(p(\phi \cdot f)) = j^l \psi \cdot j^l(\phi \cdot f)$. That is, $G_{\phi \psi} \equiv \psi G_\phi \psi^ {-1}$ is a compact subgroup of $\A_{\phi \cdot f}$ such that $j^l G_{\phi \psi} = \psi^l g_\phi (\psi^l)^{-1}$ is a maximal compact subgroup of $\A^l_{\phi \cdot f}$, where $G_\phi < \A^l_{\phi \cdot f}$ is conjugate in $\A^l$ to $G$ via $(\phi^{-1})^l$. But then, if $\tilde{G} = \phi^{-1} \psi^{-1} G_{\phi \psi} \psi \phi$, $\tilde{G}$ is a maximal compact subgroup of $\A_f$ and we have 
\[
j^l \tilde{G} = (\phi^l)^{-1} (\psi^l)^{-1} \psi^l \phi^l G (\phi^l)^{-1} (\psi^l)^{-1} \psi^l \phi^l = G.
\]
\end{proof}

Let $G$ be a maximal compact subgroup of $\A_f$. By replacing $f$ by a suitably chosen representative of its $\A$-equivalence class, we may assume that $G$ acts linearly. Given a manifold with boundary $M$, we must show that any smooth map $\alpha \colon \partial M \to \A_f/G$ extends to a smooth map $\bar{\alpha} \colon M \to \A_f/G$.

\begin{lem} \label{lift}
Suppose given a smooth map $\alpha \colon \partial M \to \A_f/G$. Then there exists a smooth lift $\tilde{\alpha} \colon \partial M \to \A_f$:
\[
\begin{diagram}
\node[2]{\A_f} \arrow{s,r}{\pi}\\
\node{\partial M} \arrow{ne,t}{\tilde{\alpha}} \arrow{e,b}{\alpha} \node{\A_f/G}
\end{diagram}
\]
\end{lem}

\begin{proof}
J\"anich \cite{janich} and Rim{\'a}nyi \cite{rimanyi} claim the existence of a lift $\tilde{\alpha}$ as a consequence of a section $\sigma \colon \A_f/G \to \A_f$ lifted from the section $\sigma^l \colon \A^l_f/G \to \A^l_f$, which exists because $\A^l_f/G$ is contractible. See the diagram (\ref{eq1}) below. We take the time to explain how such a lift is constructed. Associated with the section $\sigma^l$ is a map $\tau \colon \A^l_f \to G$ such that $\sigma^l \circ \pi^l(\alpha) = \tau(\alpha) \cdot \alpha$ for all $\alpha \in \A^l_f$. The section $\sigma$ is now given by $\sigma(\pi(g)) = \tau(j^l(g)) \cdot g$ for any representative $g$ of the orbit $\pi(g)$. The smoothness of $\tau$ is ensured by the fact that $\tau(g) = (\sigma^l ( \pi^l(\alpha)) \cdot \alpha^{-1}$. \qed
\end{proof}

\begin{equation} \label{eq1}
\begin{diagram}
\node[2]{\A_f} \arrow{e,t}{j^l} \arrow{s,l}{\pi} \node{\A^l_f} \arrow{s,r}{\pi^l}\\
\node{\A_f/G} \arrow{e,b}{\id} \arrow{ne,t,..}{\sigma?} \node{\A_f/G} \arrow{e,b}{\bar{j^l}} \node{\A^l_f/G} \node{\A^l_f/G} \arrow{w,b}{\id} \arrow{nw,t}{\sigma^l}
\end{diagram}
\end{equation}

Consider the composition $\beta=\bar{j}^l \circ \alpha \colon \partial M \to \A_f/G \to \A^l_f/G$. Since $\A^l_f/G$ is contractible, we can construct an extension $\bar{\beta} \colon M \to \A^l_f/G$. Composing with the section $\sigma^l$, we obtain a map $\gamma \colon M \to \A^l_f/G \to \A^l_f$.

\begin{lem} \label{normalization}
There exists a smooth map $\delta \colon \partial M \to G$ such that $j^l \circ (\delta \cdot \tilde{\alpha}) \colon \partial M \to \A^l_f$ coincides with $\gamma|\partial M$.
\end{lem}

\begin{proof}
Construct smooth maps
\[
\begin{array}{ll}
\delta_1 \colon \partial M \to G, & \delta_1=h \circ \gamma|\partial M,\\
\delta_2 \colon \partial M \to G, & \delta_2=h \circ \tilde{\alpha},
\end{array}
\]
and set $\delta=\delta_1 \cdot \delta_2^{-1}$. Then $j^l \circ (\delta \cdot \tilde{\alpha}) = \gamma|\partial M$ if and only if 
\begin{equation} \label{phi}
\varphi \circ j^l \circ (\delta \cdot \tilde{\alpha}) = \varphi \circ \gamma|\partial M,
\end{equation}
and it is clear from the commutative diagram below that $\pr_{\A^l_f/G} \circ \varphi \circ j^l \circ (\delta \cdot \tilde{\alpha}) =\pi^l \circ j^l \circ (\delta \cdot \tilde{\alpha}) = \pi^l \circ j^l \circ \tilde{\alpha} = \bar{\beta}|\partial M = \pr_{\A^l_f/G} \circ \varphi \circ \gamma|\partial M$,
\[
\begin{diagram}
\node[2]{\A_f} \arrow{e,t}{j^l} \arrow{s,r}{\pi} \node{\A^l_f} \arrow{se,t}{\varphi} \arrow{s,,l}{\pi^l} \arrow{e,t}{h} \node{G}\\
\node{\partial M} \arrow{ne,t}{\tilde{a}} \arrow{e,b}{\alpha} \arrow{s} \node{\A_f/G} \arrow{e,b}{\bar{j}^l} \node{\A^l_f/G} \node{G \times \A^l_f/G} \arrow{n,r}{\pr} \arrow{w,b}{\pr}\\
\node{M} \arrow{nee,b}{\bar{\beta}}
\end{diagram}
\]
while $\pr_G \circ \varphi \circ j^l \circ (\delta \cdot \tilde{\alpha}) = h \circ j^l \circ (\delta \cdot \tilde{\alpha}) = \delta \cdot (h \circ j^l \circ \tilde{\alpha}) = (h \circ \gamma|\partial M ) \cdot (h \circ j^l \circ \tilde{\alpha})^{-1} \cdot (h \circ j^l \circ \tilde{\alpha}) = h \circ \gamma| \partial M
= \pr_G \circ \varphi \circ \gamma|\partial M$. Thus (\ref{phi}) is true, and this concludes the proof of lemma~\ref{normalization}. \qed
\end{proof}

By lemma~\ref{normalization} we see that replacing the old map $\tilde{\alpha}$ by $\delta \cdot \tilde{\alpha}$, we may assume that $j^l \circ \tilde{\alpha} = \gamma| \partial M$. This will enable us to construct a map $\bar{\bar{\alpha}} \colon M \to \A$ which extends $\tilde{\alpha}$. Without the assumption $j^l \circ \tilde{\alpha} = \gamma| \partial M$ we risk -- for instance, if $M=[0,1]$ -- that $j^l(\tilde{\alpha}(0))$ and $j^l(\tilde{\alpha}(1))$ end up in different components of $\A^l_f$, in which case an extension of $\tilde{\alpha}$ is impossible. This, however, is not crucial in order to get the map into $\A_f/G$.

As a first step, we construct a map $\alpha' \colon M \to \A_{j^lf}=\{\phi \in \A|j^l\phi \in \A^l_f\}$, extending $\tilde{\alpha}$, and in particular such that $\pi_{\A_{j^lf}} \circ \alpha'$ extends our map $\partial M \stackrel{\alpha}{\to} \A_f/G \to \A_{j^l f}/G$, where the second map is induced by the inclusion, and such that $j^l \alpha' = \gamma$.

In order to do this, we must construct diffeomorphism germs at $M \times \{0\}$:
\[
\begin{array}{c}
F_1 =\bigsqcup_{k=1}^s F_1^k \colon \left( \bigsqcup_s M \times \R^n, \bigsqcup_s M \times 0 \right) \to \left( \bigsqcup_s M \times \R^n, \bigsqcup_s M \times 0 \right)\\
F_2 \colon \left( M \times \R^p, M \times 0 \right) \to \left( M \times \R^p, M \times 0 \right)
\end{array}
\]
from given germs in $\partial M \times \bigsqcup_s \R^n$ and $\partial M \times \R^p$, and where the $l$-jets are given everywhere.

We go through the construction for one of the $F^k_1$; the proof for $F_2$ is similar. If we can find $F^k_1$ locally, then we get a global solution by using a partition of unity to add the solutions together fiberwise. This gives a diffeomorphism germ since the $l$-jet, and thus in particular the differential $DF^k_1$, is fixed everywhere.

At points in the interior of $M$, we can just define the local $F^k_1$ by taking the $l^{\textrm{th}}$ degree polynomial representative of the given jet. Near points in $\partial M$, we construct the component functions of the local $F^k_1$ in the following way:

Given a polynomial $P$ of degree $l$ in the variables $x_1, \ldots, x_n$ with coefficients from the ring $\mathscr{E}(r)$ of smooth functions in $r$ variables, and a smooth function $p_0 \colon \R^{r-1+n} \to \R$ such that $j^l_x p_0=P(0, u_2, \ldots, u_r, x_1, \ldots, x_n)$, (here the coordinates of $\R^n$ are denoted by $x_i$ and the local coordinates of $M$ are denoted by $u_j$, where $\partial M$ is given by $u_1=0$) we construct a smooth function $p \colon \R^{r+n} \to \R$ such that $p|\{u_1=0\}=p_0$, and $j^l_x p=P(u_1, \ldots, u_r, x_1, \ldots, x_n)$. One function which satisfies all of the above, is the combination $p(u,x)=p_0(u_2, \ldots, u_r,x)-P(0, u_2, \ldots, u_r, x)+P(u_1, \ldots, u_r, x)$.

It follows that the map $\alpha' \colon M \to \A_{j^lf}$ exists, and it is represented by the product $F=(F^1_1, \ldots F_1^s, F_2)$. We use it to construct the extension sought $\bar{\alpha} \colon M \to \A_f/G$:

We compare the maps
\[
\phi_\nu \colon M \times \bigsqcup_s \R^n \to M \times \R^p \quad \nu = 1,2
\]
given by
\[
\begin{array}{l}
\phi_1 \colon (u,x) \mapsto \left(u, (\alpha'(u) \cdot f)(x)\right),\\
\phi_2 \colon (u,x) \mapsto \left(u, f(x)\right).
\end{array}
\]
These two maps coincide on $\partial M \times \bigsqcup_s \R^n$, and their $l$-jets coincide at each $M$-level, thus we can apply proposition~\ref{extensionprop} to find a smooth map $\psi \colon M \to \A$ such that $\psi \cdot \phi_1=\phi_2$, $j^1 \psi = \id$ and such that $\psi|\partial M=\id$. Now the map $\bar{\bar{\alpha}} \colon M \to \A_f, \quad \bar{\bar{\alpha}} = \psi \cdot \alpha'$, is an extension of $\tilde{\alpha}$ and $j^l \bar{\bar{\alpha}}=\gamma$ (because $j^1 \psi=\id$). Most importantly, it defines an extension $\bar{\alpha}=\pi \circ \bar{\bar{\alpha}} \colon M \to \A_f/G$ of $\alpha$. \qed
\end{proof}

From the proof of theorem~\ref{contractibilitythm}, we see that the following corollary also holds:

\begin{cor} \label{coverextension}
Suppose that $G$ is a maximal compact subgroup of $\A_f$. Given a smooth map $\alpha \colon \partial M \to \A_f$, we can find a smooth map $\gamma \colon \partial M \to G$ and a smooth map $\tilde{\alpha} \colon M \to \A_f$ such that $\tilde{\alpha}|\partial M = \gamma \cdot \alpha$. \qed
\end{cor}

\section{The structure of $MC(\A_f)$}

\subsection{Factorization of $MC(\A_f)$ for multigerms}

The main goal of this section is to prove that maximal compact subgroups of $\A_f$ for multigerms $f$ decompose into products of maximal compact subgroups of $\A_{g_i}$ for minimal representatives $g_i$ of their associated monogerms:

\begin{thm} \label{multitomono}
We are given a ministable multigerm
\[
f=f_1 \sqcup \ldots \sqcup f_s \colon (\R^n,0) \sqcup \ldots \sqcup (\R^n,0) \to (\R^p,0),
\]
where 
\begin{equation} \label{standardform}
f_i=\sigma_i \circ (g_i \times \emph{\id}) \colon \R^{n_i} \times \prod_{j=1, j \neq i}^s \R^{p_j} \stackrel{g_i \times \emph{\id}}{\longrightarrow} \R^{p_i} \times \prod_{j=1, j \neq i}^s \R^{p_j} \stackrel{\sigma_i}{\rightarrow} \prod_{j=1}^s \R^{p_j},
\end{equation}
where $g_i$ is a ministable unfolding of a rank $0$ representative $h_i \colon \R^{\tilde{n}_i} \to \R^{\tilde{p}_i}$ of $f_i$, and where $n=n_i+\sum_{j=1, j \neq i}^s p_j$ and $p=\sum_{j=1}^s p_j$.

Then the maximal compact subgroup factors as
\[
MC(\A_f) \cong \prod_{j=1}^s MC(\A_{g_j}) \cong \prod_{j=1}^s MC(\K_{h_j}).
\]
\end{thm}

\begin{rem} \label{s=1rem}
When $s=1$, this is \cite[Proposition~3.2]{wallsymm}.
\end{rem}

The proof of this theorem rests on relations between maximal compact subgroups of $\K$- and $\A$-equivalences, the monogerm versions of which are well known. The following lemma is analogous to \cite[Theorem~1.1]{rimanyi}.

\begin{lem} \label{lem8}
Let $f \colon \bigsqcup_s (\R^n,0) \to (\R^p,0)$ have rank $0$ and let $F \colon \bigsqcup_s (\R^N,0) \to (\R^P,0)$ be a ministable unfolding of $f$. Suppose that $G < \A_F$ is a compact subgroup. Then there exists a compact subgroup of $\K_f$ which is isomorphic to $G$. In particular, $MC(\A_F) < MC(\K_f)$.
\end{lem}

\begin{proof}
Since $G$ is compact, there exists $h \in \A$ such that $G_0=hGh^{-1}$ is linear. If $F_0=h \cdot F$, then $G_0 < \A_{F_0}$. We shall construct a map $f_0 \sim_{\K} f$ such that $F_0$ is a ministable unfolding of $f_0$.

Denote by $\Gamma(F_0)$ the graph of $F_0$. If we write $F_0=\bigsqcup_{k=1}^s F_0^k$, then 
\[
\Gamma(F_0)= \bigsqcup_{k=1}^s \Gamma(F_0^k)=\bigsqcup_{k=1}^s \{(x, F^k_0(x))|x \in \R^N\}. 
\]
Since $G_0$ is linear and $\Gamma(F_0)$ is $G_0$-invariant, the tangent space 
\[
T_0\Gamma(F_0) \subset (\R^N \sqcup \ldots \sqcup \R^N) \times \R^P
\] 
is also $G_0$-invariant. 

Define subspaces 
\[
\begin{array}{ll}
A^k & :=T_0(\Gamma(F^k_0)) \cap \left(\R^N \times \{0\}\right),\\
C^k & :=\pr_{\R^P}(T_0(\Gamma(F^k_0))),
\end{array}
\] 
which are also $G_0$-invariant. Choose $G_0$-invariant complements $B^k$ and $D^k$ of $A^k$ and $C^k$ in $\R^N$ and $\R^P$, respectively. Then $A^k \cong \R^n$, $B^k \cong \R^r$, $C^k \cong \R^r$ and $D^k \cong \R^p$.

Denote by $f_0$ the map germ
\[
\bigsqcup_{k=1}^s f_0^k \colon \bigsqcup_s \R^n \to \R^p,
\]
where $f_0 = (\pr_{D^k} \circ F^k_0)| A^k$ for each $k = 1 , \ldots, s$. We now prove that through its action on $\bigsqcup_{k=1}^s \R^N \times \R^P$, $G_0$ is a subgroup of $\K_{f_0} = \K_{f_0^1} \times \cdots \times \K_{f_0^s}$.

\subsubsection*{$F_0$ is a ministable unfolding of $f_0$}

It follows from the definition of $C^k$ that $\pr_{C^k} \circ F_0$ is a submersion for each $k$, so in particular the inclusion of $D^k$ in $\R^P$ is transverse to the germ $F^k_0$. Furthermore, the diagram
\[
\begin{diagram}
\node{A^k} \arrow{e,t}{f^k_0} \arrow{s} \node{D^k} \arrow{s} \arrow{e} \node{\R^p} \arrow{s}\\
\node{A^k \times B^k} \arrow{e,b}{F^k_0} \node{D^k \times C^k} \arrow{e} \node{\R^p \times \R^r}
\end{diagram}
\]
is trivially Cartesian for each $k$, and it follows that  $F_0$ unfolds $f_0$. The unfolding $F_0$ is ministable because it is $\A$-equivalent to the ministable map $F$.

Note, moreover, that since $F_0$ and $F$ are $\A$-equivalent, the maps $f$ and $f_0$ must be $\K$-equivalent. 

Project the group
\[
G_0 < \prod_{k=1}^s GL(A^k) \times GL(B^k) \times GL(C^k) \times GL(D^k)
\]
onto
\begin{equation} \label{projectiontarget}
\prod_{k=1}^s GL(A^k) \times GL(D^k);
\end{equation}
then the resulting group lies in $\A_{f_0^1} \times \cdots \times \A_{f_0^s}$, so in particular it must lie in $\K_{f_0}$.

\subsubsection*{The projection restricts to an injection on $G_0$} 

Equivalently, the action of $G_0$ on the $A^k$ and on $D^k$ determines the action on the $B^k$ and $C^k$. The actions on $A^k$, $C^k$ and $D^k$ determines that on $B^k$ (being a germ, $F_0$ is level-preserving if the coordinates on $B^k$ are appropriately chosen), thus it is enough to show that the action on $C^k$ is determined by that on the $A^k$ and $D^k$.

For a general ministable multigerm 
\[
\eta=\bigsqcup_{k=1}^s \eta_k \colon \bigsqcup_{k=1}^s (\R^N,0) \to (\R^P,0),
\]
define $\mathscr{N}_\eta = \theta_\eta/t \eta(\bigoplus_s \theta_{(\R^N,0)}) + \eta^\ast (m(P))\theta_\eta$ \cite{ma3}. Each $\eta_k$ is $\A$-equivalent to a germ $\tilde{\eta}_k \times \id_{\R^{d_k}}$, so that $\tilde{\eta}_k \colon (\R^{n_k},0) \to (\R^{p_k},0)$ is ministable and has an isolated singularity.

We can decompose into monogerm components: 
\[
\begin{array}{l}
\theta_\eta \cong \bigoplus_{k=1}^s \theta_{\eta_k}\\
t\eta(\bigoplus \theta_{(\R^n,0)}) \cong \bigoplus_{k=1}^s t \eta_k(\theta_{(\R^n,0)})\\
\eta^\ast (m(p))\theta_\eta = \{f \circ \eta | f \colon (\R^p,0) \to (\R,0)\} \cdot \theta_\eta \cong \bigoplus_{k=1}^s \eta_k^\ast(m(p))\theta_{\eta_k}
\end{array}
\]
and see that 
\[
\begin{array}{ll}
\mathscr{N}_\eta & \cong \bigoplus_{k=1}^s \theta_{\eta_k}/\bigoplus_{k=1}^s \left( t\eta_k(\theta_{(\R^n,0)}) + \eta_k^\ast (m(p))\theta_{\eta_k} \right)\\
 & \cong \bigoplus_{k=1}^s \theta_{\eta_k}/\left( t \eta_k(\theta_{(\R^n,0)}) + \eta_k^\ast (m(p))\theta_{\eta_k} \right) = \bigoplus_{k=1}^s \mathscr{N}_{\eta_k} \cong \bigoplus_{k=1}^s \mathscr{N}_{\tilde{\eta}_k}.
 \end{array}
\]

In our situation, we write $F_0=\bigsqcup_{k=1}^s F_0^k$, where $F_0^k=\sigma_k \circ \left( \tilde{F}^k_0 \times \id_{\R^{P-p_k}} \right)$ and $\tilde{F}^k_0 \colon \R^{n_k} \to \R^{p_k}$.

In the diagram
\[
\begin{diagram}
\node{\R^P=\bigoplus_{k=1}^s \R^{p_k} = \bigoplus_{k=1}^s \left( \theta_{(\R^{p_k},0)}/m(p_k)\theta_{(\R^{p_k},0)} \right)} \arrow{s,r}{wF_0= \bigoplus wF^k_0}\\ \node{\mathscr{N}_{F_0}=\bigoplus_{k=1}^s \mathscr{N}_{F^k_0}} \node{\bigoplus_{k=1}^s \mathscr{N}_{f_0^k} = \mathscr{N}_{f_0}} \arrow{w,t}{q_{F_0,f_0}}
\end{diagram}
\]
there is a naturally defined $G_0$-action on each of the spaces, and both maps $wF_0$ and $q_{F_0,f_0}$ are $G_0$-equivariant. Hence, the action on $\mathscr{N}_{f_0}$ determines that on $\R^P$, and in particular that on the $C^k$. It follows that the projection (\ref{projectiontarget}) is injective on $G_0$.

Hence, $G_0$ can be viewed as a compact subgroup of $\K_{f_0}$. Since the germs $f$ and $f_0$ are $\K$-equivalent, the compact subgroups of $\K_f$ and $\K_{f_0}$ are conjugate in $\K$. There is thus a compact subgroup of $\K_f$ which is conjugate in $\K$, and thus isomorphic, to $G_0$, which again is isomorphic to $G$. This concludes the proof of lemma~\ref{lem8}. \qed
\end{proof}

\begin{proof}[Proof of theorem~\ref{multitomono}]
Denote by $h$ the rank $0$ multigerm $\bigsqcup_{i=1}^s h_i$. Now we merely put the pieces together:
\[
\begin{array}{lcl}
MC(\A_f)  & \stackrel{\textrm{lemma~\ref{lem8}}}{<} & MC(\K_h)\\
 & \stackrel{\textrm{p.~\pageref{unionK}}}{\cong} &  \prod_{i=1}^s MC(\K_{h_i})\\
 & \stackrel{\textrm{remark~\ref{s=1rem}}}{\cong} & \prod_{i=1}^s MC(\A_{g_i})\\
 & < & MC(\A_f),
\end{array}
\]
where the last inequality is most easily seen to hold by considering the form (\ref{standardform}) and taking each $((\psi_i, \phi_i)_{i=1}^s) \in \prod_{i=1}^s MC(\A_{g_i})$ to the element 
\[
\left(\bigsqcup_{i=1}^s \phi_1 \times \ldots \times \psi_i \times \ldots \times \phi_s, \phi_1 \times \ldots \times \phi_s\right).
\]
This concludes the proof of the theorem. \qed
\end{proof}

\subsection{A remark on the decomposition of $MC(\A_f)$ in terms of $MC(\mathscr{R}_f)$} \label{splitexactresult}

For monogerms $f$, du Plessis and Wilson~\cite{dpwil} have studied decomposition of $MC(\A_f)$ in terms of $MC(\mathscr{R}_f)$ and a subgroup of target diffeomorphisms which preserve the discriminant of $f$, that is, the set $D(f) = f(\Sigma(f))$. The monogerms studied in~\cite{dpwil} are \emph{critical normalizations}, a class of germs which contains all stable map-germs, finitely $\A$-determined germs with $p \ge 3$ and $\Sigma(f) \neq \{0\}$, and all analytic, topologically stable map-germs. We are interested in \cite[Theorem~1.5]{dpwil}, which says that if $f$ is a finitely $\A$-determined critical normalization, then a certain subgroup, $\tilde{\textrm{Inv}}(D)$, of the target diffeomorphisms which preserve the discriminant, admits a maximal compact subgroup $G$. Moreover, the natural sequence
\begin{equation} \label{splitseq}
0 \to MC(\mathscr{R}_f) \to MC(\A_f) \to G \to 0
\end{equation}
should be split exact. As a consequence, the maximal compact subgroup $MC(\A_f)$ should be the direct sum of $MC(\mathscr{R}_f)$ and $G$. However, a closer inspection revealed a problem with the proof of the monogerm version from~\cite{dpwil}. 

Our goal is to prove the following conjecture, generalizing the claim by du Plessis and Wilson to multigerms:

\begin{conj} \label{conjecture}
Let $f$ be a finitely $\A$-determined multigerm which is also a critical normalization. The group $\tilde{\textrm{\emph{Inv}}}(D(f))$ admits a maximal compact subgroup $G$, which is unique up to conjugation. Moreover, the sequence 
\begin{equation}
0 \to MC(\mathscr{R}_f) \stackrel{i}{\to} MC(\A_f) \stackrel{p}{\to} G \to 0
\end{equation}
is split exact, so that $MC(\A_f) = MC(\mathscr{R}_f) \oplus G$. \qed
\end{conj}

In this section we shall discuss why the original proof of the monogerm version of the conjecture does not hold. Moreover, we shall see just how close to the result sought we can get with tactics similar to those in~\cite{dpwil}.

The proof of \cite[Theorem~1.5]{dpwil} rests on the assumption that linearizability is preserved under conjugation by any element of $\A$, but unfortunately, this is not true. More precisely, the problem is found in the following sentence in \cite[lines 2-4, page 270]{dpwil}: \emph{''... $l$ conjugates $\tilde{\textrm{Inv}}(D(f))$ onto $\tilde{\textrm{Inv}}(D(g))$ and conjugates compact Lie subgroups to compact Lie subgroups, preserving conjugates and preserving linearizability.''} Here, $(r, l)$ is an arbitrary element of $\A$, and \emph{linearizability} of a compact Lie subgroup $H$ of $\tilde{\textrm{Inv}}(D(f))$ means that for some $\tilde{l} \in \tilde{\textrm{Inv}}(D(f))$ with $j^1\tilde{l} = \id$, $\tilde{l} H \tilde{l}^{-1}$ is linear. However, this statement does not generally hold:

\begin{prop} \label{nothold}
For a general element $(r, l) \in \A$, the conjugate group $l H l^{-1}$ is not generally linearizable.
\end{prop}

\begin{proof}
To see this, consider the following linearizable diffeomorphism germ $\phi$:
\[
\begin{diagram}
\node{(\R^p, 0)} \arrow{e,t}{\phi} \arrow{s,l}{\psi} \node{(\R^p,0)} \arrow{s,r}{\psi}\\
\node{(\R^p,0)} \arrow{e,b}{T_0\phi} \node{(\R^p,0)}
\end{diagram}
\]
and suppose that $l \colon (\R^p, 0) \to (\R^p, 0)$ is another diffeomorphism-germ. The question now is, whether we can find a diffeomorphism-germ $\tilde{\psi} \in \tilde{\textrm{Inv}}(D(f))$ such that $T_0\tilde{\psi} = \id$ and the following diagram commutes?
\[
\begin{diagram}
\node{(\R^p, 0)} \arrow{e,t}{l\phi l^{-1}} \arrow{s,l}{\tilde{\psi}} \node{(\R^p,0)} \arrow{s,r}{\tilde{\psi}}\\
\node{(\R^p,0)} \arrow{e,b}{T_0(l \phi l^{-1})} \node{(\R^p,0)}
\end{diagram}
\]
Decomposing the diagram, we get
\[
\begin{diagram}
\node{(\R^p, 0)} \arrow{e,t}{l^{-1}} \arrow{s,l}{T_0 l \circ \psi \circ l^{-1}} \node{(\R^p, 0)} \arrow{e,t}{\phi} \arrow{s,l}{\psi} \node{(\R^p,0)} \arrow{s,r}{\psi} \arrow{e,t}{l} \node{(\R^p, 0)} \arrow{s,l}{T_0 l \circ \psi \circ l^{-1}}\\
\node{(\R^p,0)} \arrow{e,b}{T_0l^{-1}} \node{(\R^p,0)} \arrow{e,b}{T_0\phi} \node{(\R^p,0)} \arrow{e,b}{T_0l} \node{(\R^p,0)}
\end{diagram}
\]
giving $\tilde{\psi}= T_0 l \circ \psi \circ l^{-1}$. It is easy to see that $T_0\tilde{\psi} = \id$, but we also need $\tilde{\psi}$ to belong to $\tilde{\textrm{Inv}}(D(f))$. In particular, we need $\tilde{\psi}(D(f)) = D(f)$. Since $l$ and $\psi$ both belong to $\textrm{Inv}(D(f))$, we have $\psi(l^{-1}(D(f)) = D(f)$. However, $\tilde{\psi}(D(f)) = T_0 l(\psi(l^{-1}(D(f)))) = T_0l(D(f))$, which equals $D(f)$ if and only if $D(f) = T_0D(f)$, but this is not generally the case. It follows that linearizability is not preserved under conjugation. \qed
\end{proof}

As a consequence of proposition~\ref{nothold}, the proof of \cite[Theorem~1.5]{dpwil} does not hold.

The linearizability is used in~\cite{dpwil} to relate compact subgroups of $\tilde{\textrm{\textrm{Inv}}}(D(f))$ to compact subgroups of $\A_f$. Let us temporarily assume that there is an alternative way to do this. We can prove:

\begin{prop}
Let $f$ be a multigerm as in conjecture~\ref{conjecture}. Suppose that for any linear compact subgroup $H_0$ of $\tilde{\emph{Inv}}(D(f))$, we can find a compact subgroup $\tilde{G}$ of $\A_f$ such that $p(\tilde{G}) \supset H$. Then $p(MC(\A_f))$ is a maximal compact subgroup of $\tilde{\textrm{\emph{Inv}}}(D(f))$, and the sequence (\ref{splitseq}) is split exact, so that $MC(\A_f) = MC(\mathscr{R}_f) \oplus p(MC(\A_f))$.
\end{prop}

\begin{proof}
The following lemma holds, even for multigerms, with a proof similar to that in \cite{dpwil}:

\begin{lem} \label{41lemma}
Let $f = \bigsqcup_s f_i \colon \bigsqcup_s (\R^n,0) \to (\R^p,0)$ be a CN. Then the sequence
\begin{equation}
1 \to \mathscr{R}_f \stackrel{i}{\to} \A_f \stackrel{p}{\to} \tilde{\textrm{\emph{Inv}}}(D) \to 1
\end{equation}
is well-defined and exact, where $i(r) = (r, \textrm{\emph{id}})$ and $p(r, l) = l$. \qed
\end{lem}

If $f$ is a finitely $\A$-determined multigerm, we know by theorem~\ref{existenceofmc} that $MC(\A_f)$ exists and is unique up to conjugation in $\A_f$. Just as in \cite{dpwil}, also $MC(\mathscr{R}_f)$ exists and is unique up to conjugation in $\mathscr{R}_f$. Moreover, the sequence
\[
1 \to MC(\mathscr{R}_f) \stackrel{i}{\to} MC(\A_f) \stackrel{p}{\to} p(MC(\A_f)) \to 1
\]
is split exact, where $p \colon \A \to \mathscr{L}$. It is clear that $p(MC(\A_f))$ is a compact subgroup of $\textrm{Inv}(D)$; what we would like to prove is that $p(MC(\A_f))$ is a maximal compact subgroup of $\tilde{\textrm{Inv}}(D)$, unique up to conjugation.

Let $G$ be a maximal compact subgroup of $\A_f$, and note that if we can prove the result with $f$ replaced by an $\A$-equivalent germ $f_0$, then the result holds also for $f$.

Pick a compact subgroup $H$ of $\tilde{\textrm{Inv}}(D)$. We would like to show that $p(G)$ is a maximal compact subgroup of $\tilde{\textrm{Inv}}(D)$ by showing that $H$ is conjugate in $\tilde{\textrm{Inv}}(D)$ to a subgroup of $p(G)$. Choose a germ $f_0 = \alpha \cdot f$ which is $\A$-equivalent to $f$, where $\alpha = (r_1, \ldots, r_s,l) \in \A$, such that $H_0 = lHl^{-1}$ is linear. By abuse of notation, we identify $H_0$ with $j^kH_0$ for any $k$.

If we can find a maximal compact subgroup $\tilde{G}$ of $\A_{f_0}$ such that $H_0 \subset p(\tilde{G})$, then we can prove the result as follows: Since $\tilde{G}$ and $G$ are maximal compact subgroups of $\A_{f_0}$ and $\A_f$, respectively, there exists some $g = (\tilde{r}_1, \ldots, \tilde{r}_s, \tilde{l}) \in \A_{f_0}$ such that $g \tilde{G} g^{-1} = G$, and $p(G) = p(g\tilde{G}g^{-1}) = \tilde{l}p(\tilde{G}) \tilde{l}^{-1}$. Moreover, $\tilde{l} \in \tilde{\textrm{Inv}}(D(f_0))$ since $g \in \A_{f_0}$, and $p(G) = \tilde{l} p(\tilde{G}) \tilde{l}^{-1}$, so $H_0$ is conjugate to a subgroup of $p(G)$ via $\tilde{l}$. But then $H$ is conjugate to a subgroup of $p(G)$ via $\tilde{l}l$. \qed
\end{proof}

We are thus left with the problem of lifting the group $H_0$ to a compact group $\tilde{G} < \A_{f_0}$. Let us attack the problem on the jet level. 

\begin{prop}
Let $f$ be as in conjecture~\ref{conjecture}, and let $H_0$ be a compact, linear subgroup of $\tilde{\textrm{Inv}}(D(f_0))$. For any $r \in \N_0$, we can find a compact group $\tilde{G} < \A_{f_0}$ such that $j^rp \tilde{G} = H_0$.
\end{prop}

\begin{proof}
Since $H_0$ acts linearly, we have $H_0 = j^k H_0$ for any $k \ge 1$, and by the arguments above, $H_0 \subset \im (j^k p)$, where $p \colon \A_f \to \textrm{L}$. By \cite[Corollary~4.4]{dpwil}, if  $h \colon H \to H'$ is a surjective homomorphism of real algebraic groups, then every compact subgroup of $H'$ is the image under $h$ of a compact subgroup of $H$. Hence, there exists a compact subgroup $G_0$ of $\A^k_{j^kf_0}$ such that $j^kp(G_0) = H_0$. 

We would like to find a corresponding subgroup on the map level, namely a subgroup $\tilde{G}_0$ of $\A_{f_0}$ with $j^k \tilde{G}_0 = G_0$ and $p(\tilde{G}_0) \supset H_0$. By Bochner's linearization theorem, we may (by changing source coordinates) assume that $G_0$ acts linearly, while we still have $j^kp(G_0) = H_0$. Let $f'$ denote the polynomial representative of $j^kf$. Since $G_0$ acts linearly, we must also have $G_0 < \A_{f'}$. Since $f'$ has the same $k$-jet as $f_0$, and these maps are finitely $\A$-determined, we have that for sufficiently large $k$, there exists $\beta = (\phi_1, \ldots, \phi_s, \psi) \in \A$ such that $f_0 = \beta \cdot f'$. Since $G_0 < \A_{f'}$, we have $\beta G_0 \beta^{-1} < \A_{f_0}$. We had $H_0 < \tilde{\textrm{Inv}}(D(f_0))$, and now we must also have $\psi H_0 \psi^{-1} < \tilde{\textrm{Inv}}(D(f))$.

As pointed out in \cite[Addendum to 3.5]{ma3}, $f_0$ is actually $k-\A_r$-determined for any $r$ with $k$ sufficiently large. Hence, we may assume that $j^r \psi = \id$ for any $r$, and for sufficiently large $k$ (depending on $r$). Thus there exist, for any $r \in \N$, a $k \gg 0$ and a $\beta = (\phi_1, \ldots, \phi_s, \psi)$ such that the compact group $\tilde{G} = \beta G_0 \beta^{-1} < \A_{f_0}$ satisfies $j^kp(\tilde{G}) = j^k(\psi H_0 \psi^{-1}) = H_0$. We may assume $r \le k$, so $j^rp(\tilde{G}) H_0$, and we have $j^rp(\tilde{G}) = j^r\psi H_0 (j^r\psi)^{-1} < \tilde{\textrm{Inv}}(D(f_0))$. \qed
\end{proof}

This shows that on the jet level, we can get arbitrarily close to the lifting of $H_0$. Unfortunately, this is not immediately enough to get a lifting of $H_0$ to $\A_{f_0}$, but leaves us with the following conjecture, which will be considered further in a forthcoming paper:

\begin{conj}
There exists a compact subgroup $\tilde{G}$ of $\A_f$ such that $p(\tilde{G}) \supset H_0$. \qed
\end{conj}

If this conjecure holds, then by the argument above, so does conjecture~\ref{conjecture}.

\subsection{Maximal compact subgroups are often small} \label{oftensmall}

Recall from lemma~\ref{basicobservations} that $\A_f$ is a subgroup of $\K_f$, so the size of $\K_f$ is an upper bound for the size of $\A_f$. In this section we prove that $\K_f$ is very small for finitely determined rank $0$ germs, making it easy to compute. Since we have shown that maximal compact subgroups of $\K_f$ for multigerms $f$ can be decomposed as a product of maximal compact subgroups for the corresponding monogerms, the monogerm results carry directly over to multigerms.

\begin{thm} \label{maxcomp}
Let $f \colon (\R^n,0) \to (\R^p,0)$ be finitely $\K$-determined, with $p<n$ and $T_0f \equiv 0$. If $p > 1$, or if $p=1$ and $j^2 f = 0$, then $MC(\K_f)$ is $\le 1$-dimensional, and if $p=1$ then it is $0$-dimensional. 
\end{thm}

For $p=1$, this is related to a theorem by P.~Slodowy:

\begin{thm} \emph{\cite[Satz~p.~169]{slodowy}} \label{slodowy}
Let $f \colon (\R^n, 0) \to (\R,0)$  be a germ such that $j^2 f = 0$ and $f$ is finitely $\mathscr{R}$-determined. If a compact group $G$ acts faithfully and linearly on $\R^n$, leaving $f$ invariant, then $G$ is zero-dimensional. \qed
\end{thm}

\begin{rem}
By \cite[Theorem~4.6.1]{wallfindet}, any $\K$-finitely determined function germ is $\mathscr{R}$-finitely determined.
\end{rem}

\begin{cor}
Let $f$ be as in theorem~\ref{maxcomp} with $p=1$, and let $G < \mathscr{R}_f$ be a compact subgroup. Then $G$ is zero-dimensional. \qed
\end{cor}

\begin{thm}
Let $f$ be as in theorem~\ref{maxcomp} with $p=1$, and let $G < \K_f$ be a compact subgroup. Then $G$ is zero-dimensional.
\end{thm}

\begin{proof}
Changing $f$ by a $\K$-equivalence, we may assume that $G$ acts linearly by lemma~\ref{basicobservations}. Note furthermore that changing $f$ by a $\K$-equivalence will not change the fact that $j^2f=0$.

Linear subgroups of $\K$ lie in $\A$ by lemma~\ref{basicobservations}; hence we can assume $G < GL_n \times GL_1=GL_n \times \R^\ast$. The projections from $GL_n \times \R^\ast$ onto $GL_n$ and $\R^\ast$ are continuous homomorphisms, and take $G$ to compact subgroups $\tilde{G} <  GL_n$ and $G_\R < \R^\ast$, respectively. Since $G_\R$ is a compact subgroup of $\R^\ast$, we must have $G_\R < \{\pm 1\}$.

Having this in mind, we see that $\tilde{G}$ splits into two parts, namely
\[
\tilde{G} \cap \mathscr{R}_f \quad \textrm{ and } H=\{g \in \tilde{G}|g \cdot f = -f\}.
\]
The group $\tilde{G} \cap \mathscr{R}_f$ is finite by theorem~\ref{slodowy}, but what about $H$? 

Since $\tilde{G}$ is a Lie group, we must either have $\tilde{G}$ discrete, or $\tilde{G} \cap \mathscr{R}_f \subset \partial_{\tilde{G}} H$ with $\dim H \ge 1$. Suppose the latter. Then we can form a continuous path $\gamma \colon I \to \tilde{G}$ such that $\gamma(0) \in \mathscr{R}_f \cap \partial_{\tilde{G}} H$ and $\gamma(t) \in H$ for $t \neq 0$. Then we have $\gamma(t) \stackrel{t \to 0}{\to} \gamma(0)$, and for any given $x \in (\R^n,0)$ we have $f(\gamma(t)(x)) \stackrel{t \to 0} {\to}f(\gamma(0)(x))$, since $G$ is a matrix group. But by the definitions of $H$ and $\mathscr{R}_f$, we have $f(\gamma(t)(x))=-f(x)$ when $t \neq 0$, while $f(\gamma(0)(x))=f(x)$, so unless $f(x)=0$, this must be false. We have $f(x) \neq 0$ for $x$ arbitrarily close to $0 \in \R^n$, and hence we cannot find such a path $\gamma$. But then $\tilde{G}$ must be discrete. Being a compact discrete set, $\tilde{G}$ is finite. \qed
\end{proof}

For $p \ge 2$, C.T.C.~Wall has proven an analogous result over the complex numbers:

\begin{thm} \emph{\cite[Theorem~3.3]{wallsymm}} \label{wallthm}
Let $f \colon (\C^n,0) \to (\C^p,0)$ have finite singularity type, $1 < p < n$, and $T_0f \equiv 0$. Then
\[
\dim G_f \le 1,
\]
where $G_f$ is a maximal complex reductive subgroup of $\K_f$. \qed
\end{thm}

We shall pass from Wall's result to the corresponding claim over the real numbers from theorem~\ref{maxcomp}.

\begin{proof}[of theorem~\ref{maxcomp}]
Denote $G=MC(\K_f)$ for short. By lemma~\ref{basicobservations}, we may assume, up to a change of coordinates, that $f$ is a polynomial and that $G$ is linear. In particular,
\[
G < \K_f \cap \left( GL_n \times GL_p \right) < \A_f
\]
by lemma~\ref{basicobservations}. 

There is a corresponding complex polynomial
\[
f_\C \colon (\C^n,0) \to (\C^p,0)
\]
with the same (real) coefficients as $f$. Then $f_\C$ is finitely $\K$-determined as well \cite[Proposition~1.7]{wallfindet}, hence has FST.

Viewing $G$ as a subgroup of 
\[
GL(n, \R) \times GL(p, \R) < GL(n, \C) \times GL(p, \C),
\]
we denote by $G_\C$ the Zariski closure of $G$ in $GL(n, \C) \times GL(p, \C)$. By Schwarz \cite[2.2-2.6]{schwarz}, the set $G_\C$ is a reductive complex algebraic subgroup of the algebraic group $GL(n, \C) \times GL(p, \C)$, and if we write $\mathfrak{g}$ and $\mathfrak{g}_\C$ for the Lie algebras of $G$ and $G_\C$, respectively, then $\mathfrak{g}_\C= \mathfrak{g} + i \mathfrak{g}$.

We argue that $G_\C < \A_{f_\C}$, which will prove that $\dim_\C G_\C \le 1$ by Wall's theorem (that is, theorem~\ref{wallthm}).

The action of $G_\C$ on $(\C^n,0) \times (\C^p,0)$ is algebraic, and hence Zariski continuous. Viewing $G$ as a subset of $G_\C$ with the induced Zariski topology, the maps
\[
\begin{array}{ll}
\Phi \colon G \to \C^p, & g \mapsto (g \cdot f)_\C(z)\\
\Phi_\C \colon G_\C \to \C^p, & g_\C \mapsto (g_\C \cdot f_\C)(z)
\end{array}
\]
are Zariski continuous for any fixed $z \in \C^n$, and $\Phi_\C$ is a continuous extension of $\Phi$. 

The map $\Phi$ is constant, because $G < \A_f$ and hence $g \cdot f = f$ for all $g \in G$. But $G$ is Zariski dense in $G_\C$ by \cite{schwarz}, and points are closed in the Zariski topology on $\C^p$; hence $\Phi_\C$ must be constant as well. Since this holds for all $z \in \C^n$, it follows that $G_\C < \A_{f_\C}$, and $\dim G_\C \le 1$ by theorem~\ref{wallthm}.

Then $\dim_\C \mathfrak{g}_\C \le 1$, and since $\mathfrak{g}_\C=\mathfrak{g} + i \mathfrak{g}$, we must have $\dim_\R \mathfrak{g} \le 1$, and in particular $\dim_\R G \le 1$. \qed
\end{proof}

\subsubsection*{Example computations}

The results from section~\ref{oftensmall} allow us to efficiently compute maximal compact subgroups $\K_f$ for germs that are of particular interest to us, as illustrated in the example below. The argumentation will obviously carry over to a wide range of other cases.

\begin{ex}
Let $f \colon (\R^2,0) \to (\R,0)$ belong to the $E_{p,0}(\ast)$ or $Z_{p,0}(\ast)$-series of singularities; namely, let  $f$ be one of the map-germs
\[
\begin{array}{ll}
a) & (x,y) \mapsto x^3 + \lambda x y^{2p} + y^{3p},\\
b) & (x,y) \mapsto y(x^3 + \lambda x y^{2p} + y^{3p}),
\end{array}
\]
with $p>1$ and $\lambda \neq 0$. Then $f$ is a weighted homogeneous polynomial, and in particular, $f$ is $\R^\ast$-equivariant, so $\{\pm 1\}$ is a compact subgroup of $\A_f$. We show that $\{\pm 1\}$ is a maximal compact subgroup of $\A_f$ by showing that it is a maximal compact subgroup of $\K_f$. Below, we  give the argument for the germ $a)$; the argument for $b)$ is almost identical. See \cite[Theorem 97]{feragenthesis} for a more detailed account.

Let $G < \K_f$ be a maximal compact subgroup. Suppose that $(l, (h_1, h_2)) \in G$, which acts on $f$ by 
\[
(l, (h_1, h_2)) \cdot \tilde{f}(x,y)=l(x,y) \cdot \tilde{f}(h_1(x,y), h_2(x,y)). 
\]
We note that for elements of $\K_f$, the diffeomorphism $l$ is completely determined by $(h_1, h_2)$; in other words $G$ is determined by its action on the source space, and the projection
\[
p \colon \K=\mathscr{C} \rtimes \mathscr{R} \to \mathscr{R}, \quad (l, (h_1, h_2)) \mapsto (h_1, h_2),
\]
restricts to an injection on $\K_f$. Thus, it is enough to show that the set of pairs $h = (h_1, h_2)$, which can be part of an element in $\K_f$, is isomorphic to $\{\pm 1\}$. Since $j^1 G \cong G$, we can show this by investigating the $1$-jet of $h$, denoted
\[
j^1h=\left[\begin{array}{cc} \alpha & \beta\\ \gamma & \delta \end{array} \right].
\]

Since $h$ comes from an element of $\K_f$, we must have
\begin{itemize}
\item[a)] $\beta = 0$,
\item[b)] $\alpha, \delta \in \{ \pm 1\}$,
\item[c)] $j^1h=\left[\begin{array}{cc} 1 & 0\\ \gamma & 1 \end{array} \right]$ can only hold if $\gamma = 0$,
\item[d)] $j^1h=\left[\begin{array}{cc} -1 & 0\\ \gamma & 1 \end{array} \right]$ never holds,
\item[e)] $j^1h=\left[\begin{array}{cc} -1 & 0\\ \gamma & -1 \end{array} \right]$ can only hold if $p$ is odd,
\item[f)] $j^1h=\left[\begin{array}{cc} 1 & 0\\ \gamma & -1 \end{array} \right]$ can only hold if $p$ is even.
\end{itemize}

Here, a) follows from weighted homogeneity of $f$; b) and c) follows from the fact that $j^1(\pr_\mathscr{R}(G))$ is a finite matrix group so $h^k = \id$ for some $k \in \N_0$. To see d)-f), plug in polynomial expansions for $h_1$ and $h_2$ into the formula for $f$ and use weighted homogeneity of $f$ to deduce relations on the coefficients of the $h_i$, which can only be true under the conditions d)-f). Now we know that $G$ consists of the identity along with elements $(l, (h_1, h_2))$ such that $j^1h$ is of the form e) if $p$ is odd, or f) if $p$ is even. Again, using the fact that $G$ is finite, we see that in both cases, there is only one valid value of $\gamma$. Hence, $G \cong \Z_2$, and $\{\pm 1\} \cong \Z_2$ is a maximal compact subgroup of $\A_f$.
\end{ex}

\section*{Acknowledgements}
The authors wish to thank Hans Brodersen, Maria Aparecida Soares Ruas and David Trotman, whose careful reading detected a mistake in the original version of the argument.

This work was supported by the Magnus Ehrnrooth Foundation, as well as the Centre for Stochastic Geometry and Advanced Bioimaging, funded by a grant from the Villum Foundation. Aasa Feragen wishes to thank the Department of Mathematical Sciences at the University of Aarhus for its hospitality while the research leading to this article was carried out.

\end{document}